\documentclass[eqthmnum,nocolour]{jt-calcs}
\usepackage[backend=bibtex,style=alphabetic,sorting=nyt]{biblatex}
\bibliography{bibliography}

\usepackage{mathrsfs}

\title{On The Mackey Formula for Connected Centre Groups}
\author{Jay Taylor}
\address{Department of Mathematics, University of Arizona, 617 N. Santa Rita Ave., Tucson AZ 85721, United States.}
\email{jaytaylor@math.arizona.edu}
\mscno{2010}{20C30}{20C15}
\keywords{Finite reductive groups, Mackey formula.}

\newcommand{\Class}{\ensuremath{\mathrm{Class}}}

\begin{document}
\begin{abstract}
Let $\bG$ be a connected reductive algebraic group over $\overline{\mathbb{F}}_p$ and let $F : \bG \to \bG$ be a Frobenius endomorphism endowing $\bG$ with an $\mathbb{F}_q$-rational structure. Bonnaf\'e--Michel have shown that the Mackey formula for Deligne--Lusztig induction and restriction holds for the pair $(\bG,F)$ except in the case where $q = 2$ and $\bG$ has a quasi-simple component of type $\E_6$, $\E_7$, or $\E_8$. Using their techniques we show that if $q = 2$ and $Z(\bG)$ is connected then the Mackey formula holds unless $\bG$ has a quasi-simple component of type $\E_8$. This establishes the Mackey formula, for instance, in the case where $(\bG,F)$ is of type $\E_7(2)$. Using this, together with work of Bonnaf\'e--Michel, we can conclude that the Mackey formula holds on the space of unipotently supported class functions if $Z(\bG)$ is connected.
\end{abstract}

\section{Introduction}\label{sec:intro}
\begin{pa}\label{pa:setup}
Let $\bG$ be a connected reductive algebraic group over an algebraic closure $\overline{\mathbb{F}}_p$ of the finite field $\mathbb{F}_p$ of prime cardinality $p$. Moreover, let $F : \bG \to \bG$ be a Frobenius endomorphism endowing $\bG$ with an $\mathbb{F}_q$-rational structure, where $\mathbb{F}_q \subseteq \overline{\mathbb{F}}_p$ is the finite field of cardinality $q$. We assume fixed a prime $\ell \neq p$ and an algebraic closure $\Ql$ of the field of $\ell$-adic numbers. If $\Gamma$ is a finite group then we denote by $\Class(\Gamma)$ the functions $f : \Gamma \to \Ql$ invariant under $\Gamma$-conjugation.
\end{pa}

\begin{pa}
If $\bP \leqslant \bG$ is a parabolic subgroup of $\bG$ with $F$-stable Levi complement $\bL$ then Deligne--Lusztig have defined a pair of linear maps $R_{\bL \subset \bP}^{\bG} : \Class(\bL^F) \to \Class(\bG^F)$ and ${}^*R_{\bL \subset \bP}^{\bG} : \Class(\bG^F) \to \Class(\bL^F)$ known as Deligne--Lusztig induction and restriction. The Mackey formula, which is an analogue of the usual Mackey formula from finite groups, is then defined to be the following equality
\begin{equation}\label{eq:mackey}
{}^*R_{\bL \subset \bP}^{\bG}\circ R_{\bM \subset \bQ}^{\bG} = \sum_{g \in \bL^F \setminus \mathcal{S}_{\bG}(\bL,\bM)^F/\bM^F} R_{\bL \cap {}^g\bM \subset \bL \cap {}^g\bQ}^{\bL}\circ {}^*R_{\bL\cap {}^g\bM \subset \bP \cap {}^g\bM}^{{}^g\bM} \circ (\ad g)_{\bM^F} \tag{$\mathcal{M}_{\bG,F,\bL,\bP,\bM,\bQ}$}
\end{equation}
of linear maps $\Class(\bM^F) \to \Class(\bL^F)$, where $\bQ \leqslant \bG$ is a parabolic subgroup with $F$-stable Levi complement $\bM \leqslant \bQ$. Here
\begin{equation*}
\mathcal{S}_{\bG}(\bL,\bM) = \{g \in \bG \mid \bL\cap{}^g\bM\text{ contains a maximal torus of }\bG\}
\end{equation*}
and $(\ad g)_{\bM^F}$ is the linear map $\Class(\bM^F) \to \Class({}^g\bM^F)$ induced by the isomorphism ${}^g\bM^F \to \bM^F$ obtained by restricting the inner automorphism $(\ad g)_{\bG^F}$ of $\bG^F$ defined by conjugation with $g$.
\end{pa}

\begin{pa}
The Mackey formula is a fundamental tool in the representation theory of finite reductive groups. It's importance to ordinary representation theory is made abundantly clear in the book of Digne--Michel \cite{digne-michel:1991:representations-of-finite-groups-of-lie-type}. However it also plays a prominent role in modular representation theory via $e$-Harish-Chandra theory. The formula was first proposed by Deligne in the case where $\bP$ and $\bQ$ are both $F$-stable; a proof of this case appears in \cite[2.5]{lusztig-spaltenstein:1979:induced-unipotent-classes}. Deligne--Lusztig were also able to establish the formula when either $\bL$ or $\bM$ is a maximal torus, see \cite[Theorem 7]{deligne-lusztig:1983:duality-for-representations-II} and \cite[11.13]{digne-michel:1991:representations-of-finite-groups-of-lie-type}. We note that a consequence of the Mackey formula, namely the inner product formula for Deligne--Lusztig characters, had been shown to hold in earlier work of Deligne--Lusztig, see \cite[6.8]{deligne-lusztig:1976:representations-of-reductive-groups}.
\end{pa}

\begin{pa}
A possible approach to proving the Mackey formula is suggested by the early work of Deligne--Lusztig, see the proof of \cite[6.8]{deligne-lusztig:1976:representations-of-reductive-groups}. Here the idea is to argue by induction on $\dim \bG$. In a series of articles \cite{bonnafe:1998:formule-de-mackey,bonnafe:2000:mackey-formula-in-type-A,bonnafe:2003:mackey-formula-in-type-A-corrigenda} Bonnaf\'e made extensive progress on the Mackey formula, specifically establishing criteria that a minimal counterexample must satisfy. In fact, Bonnaf\'e was able to establish the Mackey formula assuming either that $q$ is sufficiently large (with an explicit bound on $q$) or if all the quasi-simple components of $\bG$ are of type $\A$. In the latter case Lusztig's theory of cuspidal local systems \cite{lusztig:1984:intersection-cohomology-complexes} plays a prominent role in the proofs.
\end{pa}

\begin{pa}
Using the inductive approach mentioned above, together with computer calculations performed with {\sf CHEVIE} \cite{michel:2015:the-development-version-of-CHEVIE}, Bonnaf\'e--Michel \cite{bonnafe-michel:2011:mackey-formula} were able to show the Mackey formula holds assuming either that $q > 2$ or that $\bG$ has no quasi-simple components of type $\E_6$, $\E_7$ or $\E_8$. Our contribution to this problem is to observe that the following holds.
\end{pa}

\begin{thm}\label{thm:main}
Assume that $q=2$ and $\bG$ is such that $Z(\bG)$ is connected and $\bG$ has no quasi-simple component of type $\E_8$. Then the Mackey formula \cref{eq:mackey} holds for $(\bG,F)$.
\end{thm}

\begin{pa}
Our approach to proving \cref{thm:main} is exactly the same as that of \cite{bonnafe-michel:2011:mackey-formula}; namely we argue by induction on $\dim\bG$. As remarked in \cite[3.10]{bonnafe-michel:2011:mackey-formula} to show the Mackey formula holds for all tuples $(\bG,F,\bL,\bP,\bM,\bQ)$ it is sufficient to show the Mackey formula holds when $(\bG,F)$ is of type ${}^2\E_6^{\simc}(2)$ and $\bM$ is a Levi subgroup of type $\A_2\A_2$. Our observation is that by considering the adjoint group ${}^2\E_6^{\ad}(2)$ the problematic Levi subgroup of type $\A_2\A_2$ is circumvented.
\end{pa}

\begin{pa}
In the very first step of the proof of \cite[3.9]{bonnafe-michel:2011:mackey-formula} one encounters the following problem. If $Z(\bG)$ is connected then it is not necessarily the case that $Z(C_{\bG}^{\circ}(s))$ is connected for all semisimple elements $s \in \bG$. This means one cannot apply directly, to $C_{\bG}^{\circ}(s)$, any induction hypothesis which relies on the centre being connected. However, in the cases under consideration we have enough control over the structure of $C_{\bG}^{\circ}(s)$ to make use of the induction hypothesis, see \cref{lem:cent-comps}. Let us note now that our proof of \cref{thm:main} relies on all the previously established cases of the Mackey formula.
\end{pa}

\begin{pa}
Unfortunately we cannot push our argument through to the case where $\bG^F$ is $\E_8(2)$. Here there exists a semisimple element $s \in \bG^F$ such that $C_{\bG}^{\circ}(s)^F$ is a product ${}^2\E_6^{\simc}(2)\cdot{}^2\A_2^{\simc}(2)$. Thus we arrive back to the problem of dealing with the case of ${}^2\E_6^{\simc}(2)$. However, we can establish one general statement about \cref{eq:mackey} assuming $Z(\bG)$ is connected. For this we need the following notation. Let $\bG_{\uni} \subseteq \bG$ be the variety of all unipotent elements in $\bG$. We then denote by $\Class_{\uni}(\bG^F) \subseteq \Class(\bG^F)$ the space of unipotently supported class functions of $\bG^F$, i.e., those functions $f \in \Class(\bG^F)$ for which $f(g) \neq 0$ implies $g \in \bG_{\uni}^F$.
\end{pa}

\begin{thm}\label{thm:main-unip}
Assume $Z(\bG)$ is connected then the Mackey formula \cref{eq:mackey} holds on $\Class_{\uni}(\bM^F)$.
\end{thm}

\begin{acknowledgments}
This work was carried out during a visit of the author to the TU Kaiserslautern. The author would kindly like to thank the Fachbereich Mathematik for its hospitality and the DFG for financially supporting this visit through grant TRR-195. Finally, we thank Gunter Malle for useful discussions on this work.
\end{acknowledgments}

\section{Centralisers of Semisimple Elements}\label{sec:prelim}
\begin{pa}\label{pa:closed-F-stable-subgroup}
Throughout we assume that $\bG$ and $F : \bG \to \bG$ are as in \cref{pa:setup}. In what follows we will write $\bG$ as a product $\bG_1\cdots\bG_nZ(\bG)$ where $\bG_1,\dots,\bG_n$ are the quasi-simple components of $\bG$. With this notation in place we have the following.
\end{pa}

\begin{lem}\label{lem:bij-quo-derived}
Let $\bH = \bH_1\cdots\bH_nZ(\bG) \leqslant\bG$ be an $F$-stable subgroup of $\bG$ where $\bH_i \leqslant \bG_i$ is a closed connected reductive subgroup of $\bG_i$. If $\pi : \bH \to \bH/Z^{\circ}(\bH)$ denotes the natural quotient map and $Z(\bG) \leqslant Z^{\circ}(\bH)$ then we have a bijective morphism of varieties
\begin{align*}
\pi(\bH_1) \times \cdots \times \pi(\bH_n) &\to \bH/Z^{\circ}(\bH)\\
(h_1,\dots,h_n) &\mapsto h_1\cdots h_n
\end{align*}
which is defined over $\mathbb{F}_q$. Moreover, if $Z(\bH_i) \leqslant Z^{\circ}(\bH)$ then we have $\pi(\bH_i)$ has a trivial centre.
\end{lem}

\begin{proof}
Recall that if $i \neq j$ then we have $\bG_i \cap \bG_j \leqslant Z(\bG) \leqslant Z^{\circ}(\bH)$. Hence, as $\bH_i \cap \bH_j \leqslant \bG_i \cap \bG_j$ we have $\pi(\bH_i) \cap \pi(\bH_j) = \{1\}$ which establishes the bijective morphism. Now, let us consider the case where $Z(\bH_i) \leqslant Z^{\circ}(\bH)$. We know that $\bH_i/Z(\bH_i)\leqslant \bH/Z(\bH_i)$ has a trivial centre and we have a surjective homomorphism
\begin{equation*}
\bH/Z(\bH_i) \to \bH/Z^{\circ}(\bH)
\end{equation*}
which restricts to a bijective homomorphism $\bH_i/Z(\bH_i) \to \pi(\bH_i)$. Thus $\pi(\bH_i)$ also has a trivial centre.
\end{proof}

\begin{pa}
Our application of \cref{lem:bij-quo-derived} will be to the case where $\bH$ is the connected centraliser of a semisimple element of $\bG$. Specifically we will need the following.
\end{pa}

\begin{lem}\label{lem:cent-comps}
Assume that $p=2$ and $\bG$ is such that $Z(\bG)$ is connected and all the quasi-simple components of $\bG$ are of type $\A$, $\E_6$, or $\E_7$. Then if $s \in \bG^F$ is a semisimple element there exist $F$-stable closed connected reductive subgroups $\bH_1,\bH_2 \leqslant C_{\bG}^{\circ}(s)$ with the following properties:
\begin{enumerate}[label=(\alph*)]
	\item $\bH_1$ has a trivial centre and has no quasi-simple component of type $\E_8$,
	\item all the quasi-simple components of $\bH_2$ are of type $\A$ or $\D$,
	\item there exists a bijective homomorphism of algebraic groups
\begin{equation*}
\bH_1 \times \bH_2 \to C_{\bG}^{\circ}(s)/Z^{\circ}(C_{\bG}^{\circ}(s))
\end{equation*}
	which is defined over $\mathbb{F}_q$.
\end{enumerate}
\end{lem}

\begin{proof}
As above we write $\bG$ as a product $\bG_1\cdots\bG_nZ(\bG)$ where the $\bG_i$ are the quasi-simple components of $\bG$. Similarly we may write $s$ as a product $s_1\cdots s_nz$ where $s_i \in \bG_i$ and $z \in Z(\bG)$. We then have $C_{\bG}^{\circ}(s) = C_{\bG_1}^{\circ}(s_1)\cdots C_{\bG_n}^{\circ}(s_n)Z(\bG)$, see \cite[2.2]{bonnafe:2005:quasi-isolated} for instance. By assumption each $\bG_i$ is of type $\A$, $\E_6$, or $\E_7$ which implies one of the following holds:
\begin{itemize}
	\item all the quasi-simple components of $C_{\bG_i}^{\circ}(s_i)$ are of type $\A$ or $\D$,
	\item $\bG_i$ is of type $\E_7$ and $C_{\bG_i}^{\circ}(s_i)$ is a Levi subgroup of type $\E_6$,
	\item $\bG_i = C_{\bG_i}^{\circ}(s_i)$ is of type $\E_6$ or $\E_7$.
\end{itemize}
As $p=2$ we have in the second case that $Z(\bG_i) = \{1\}$ which implies that $Z(C_{\bG_i}^{\circ}(s_i))$ is connected because $C_{\bG_i}^{\circ}(s_i)$ is a Levi subgroup of $\bG_i$. In particular, we have $Z(C_{\bG_i}^{\circ}(s_i)) \leqslant Z^{\circ}(C_{\bG}^{\circ}(s))$. In the third case we have $Z(C_{\bG_i}^{\circ}(s_i)) = Z(\bG_i) \leqslant Z(\bG) \leqslant Z^{\circ}(C_{\bG}^{\circ}(s))$ because, by assumption, we have $Z(\bG)$ is connected. The statement now follows from \cref{lem:bij-quo-derived}.
\end{proof}

\section{Around the Mackey Formula}\label{sec:con-centre-p=2}
\begin{pa}
Assume we are given a tuple $(\bG,F,\bL,\bP,\bM,\bQ)$ as in \cref{pa:setup} then we set
\begin{equation*}
\Delta_{\bL \subset \bP,\bM \subset \bQ}^{\bG} = {}^*R_{\bL \subset \bP}^{\bG}\circ R_{\bM \subset \bQ}^{\bG} - \sum_{g \in \bL^F \setminus \mathcal{S}_{\bG}(\bL,\bM)^F/\bM^F} R_{\bL \cap {}^g\bM \subset \bL \cap {}^g\bQ}^{\bL}\circ {}^*R_{\bL\cap {}^g\bM \subset \bP \cap {}^g\bM}^{{}^g\bM} \circ (\ad g)_{\bM}.
\end{equation*}
The Mackey formula \cref{eq:mackey} is therefore equivalent to the statement $\Delta_{\bL \subset \bP,\bM \subset \bQ}^{\bG} = 0$. Note that $\Delta_{\bL \subset \bP,\bM \subset \bQ}^{\bG}$ is a linear map $\Class(\bM^F) \to \Class(\bL^F)$. In what follows we will say that the Mackey formula holds for $(\bG,F)$, or for short that it holds for $\bG$, if $\Delta_{\bL \subset \bP,\bM \subset \bQ}^{\bG} = 0$ for all possible quadruples $(\bL,\bP,\bM,\bQ)$.
\end{pa}

\begin{pa}\label{pa:inf-compat}
Recall that a homomorphism $\iota : \bG \to \widetilde{\bG}$ is said to be isotypic if the following hold: $\bG$ and $\widetilde{\bG}$ are connected reductive algebraic groups, the kernel $\Ker(\iota)$ is central in $\bG$ and the image $\Image(\iota)$ contains the derived subgroup of $\widetilde{\bG}$. If $\iota$ is defined over $\mathbb{F}_q$ then this restricts to a homomorphism $\iota : \bG^F \to \widetilde{\bG}^F$ and we have a corresponding restriction map $\Res_{\bG^F}^{\widetilde{\bG}^F} : \Class(\widetilde{\bG}^F) \to \Class(\bG^F)$ defined by $\Res_{\bG^F}^{\widetilde{\bG}^F}(f) = f\circ\iota$. If $\bK \leqslant \bG$ is a closed subgroup of $\bG$ then we denote by $\widetilde{\bK}$ the subgroup $\iota(\bK)Z(\widetilde{\bG}) \leqslant \widetilde{\bG}$. With this notation we have by \cite[3.7]{bonnafe-michel:2011:mackey-formula} that
\begin{equation}\label{eq:res-formula}
\Res_{\bL^F}^{\widetilde{\bL}^F} \circ \Delta_{\widetilde{\bL} \subset \widetilde{\bP},\widetilde{\bM} \subset \widetilde{\bQ}}^{\widetilde{\bG}} = \Delta_{\bL \subset \bP,\bM \subset \bQ}^{\bG}\circ\Res_{\bM^F}^{\widetilde{\bM}^F}.
\end{equation}
The following is an easy consequence of \cref{eq:res-formula}.
\end{pa}

\begin{lem}\label{lem:bij-morph}
If $\iota : \bG \to \widetilde{\bG}$ is a bijective morphism of algebraic groups defined over $\mathbb{F}_q$ then the Mackey formula holds for $(\bG,F)$ if and only if it holds for $(\widetilde{\bG},F)$.
\end{lem}

\begin{pa}
Now assume $s \in \bG^F$ is a semisimple element then for any class function $f \in \Class(\bG^F)$ we define a function $d_s^{\bG}(f) : C_{\bG}^{\circ}(s)^F \to \Ql$ by setting
\begin{equation*}
d_s^{\bG}(f)(g) = \begin{cases}
f(sg) &\text{if $g$ is unipotent},\\
0 &\text{otherwise}.
\end{cases}
\end{equation*}
Note that $d_s^{\bG}(f) \in \Class(C_{\bG}^{\circ}(s)^F)$ so we have defined a $\Ql$-linear map $d_s^{\bG} : \Class(\bG^F) \to \Class_{\uni}(C_{\bG}^{\circ}(s)^F)$. In particular, if $z \in Z(\bG)^F$ then we obtain a $\Ql$-linear map $d_z^{\bG} : \Class(\bG^F) \to \Class_{\uni}(\bG^F)$. Now, if $s \in \bL^F$ is a semisimple element then by \cite[3.5]{bonnafe-michel:2011:mackey-formula} we have 
\begin{equation}\label{eq:d_s-formula}
d_s^{\bL} \circ \Delta_{\bL \subset \bP,\bM \subset \bQ}^{\bG} = \sum_{\substack{g \in \bG^F \\ s \in {}^g\bM}} \frac{|C_{{}^g\bM}^{\circ}(s)^F|}{|\bM^F||C_{\bG}^{\circ}(s)^F|}\Delta_{C_{\bL}^{\circ}(s)\subset C_{\bP}^{\circ}(s), C_{{}^g\bM}^{\circ}(s) \subset C_{{}^g\bQ}^{\circ}(s)}^{C_{\bG}^{\circ}(s)} \circ d_s^{{}^g\bM}\circ (\ad g)_{\bM}.
\end{equation}
Moreover, if $s \in Z(\bG)^F \leqslant \bL^F\cap\bM^F$ it follows that
\begin{equation}\label{eq:d_s-centre-formula}
d_s^{\bL} \circ \Delta_{\bL \subset \bP,\bM \subset \bQ}^{\bG} = \Delta_{\bL \subset \bP,\bM \subset \bQ}^{\bG} \circ d_s^{\bM},
\end{equation}
see \cite[3.6]{bonnafe-michel:2011:mackey-formula}.
\end{pa}

\begin{lem}\label{lem:inf-isom-uni}
Assume $\iota : \bG \to \widetilde{\bG}$ is a surjective isotypic morphism such that $\Ker(\iota) \leqslant Z^{\circ}(\bG)$ then the map $d_1^{\bG}\circ\Res_{\bG^F}^{\widetilde{\bG}^F} : \Class(\widetilde{\bG}^F) \to \Class_{\uni}(\bG^F)$ restricts to an isomorphism $\Class_{\uni}(\widetilde{\bG}^F) \to \Class_{\uni}(\bG^F)$.
\end{lem}

\begin{proof}
Note that $\iota$ restricts to a bijection $\iota : \bG_{\uni}^F \to \widetilde{\bG}_{\uni}^F$. We will denote by $\iota^{-1} : \widetilde{\bG}_{\uni}^F \to \bG_{\uni}^F$ the inverse of this map. Now, if $f \in \Class_{\uni}(\bG^F)$ then we define $\tilde{f} : \widetilde{\bG}^F \to \Ql$ by setting
\begin{equation*}
\tilde{f}(g) = \begin{cases}
f(\iota^{-1}(g)) &\text{if }g \in \widetilde{\bG}_{\uni}^F\\
0 &\text{otherwise}.
\end{cases}
\end{equation*}
The proof of \cite[3.8]{bonnafe-michel:2011:mackey-formula} shows that $u,v \in \bG_{\uni}^F$ are $\bG^F$-conjugate if and only if $\iota(u),\iota(v) \in \widetilde{\bG}_{\uni}^F$ are $\widetilde{\bG}^F$-conjugate because $\Ker(\iota) \leqslant Z^{\circ}(\bG)$ and $\iota$ is surjective. This implies $\tilde{f} \in \Class(\widetilde{\bG}^F)$ so we're done.
\end{proof}

\section{Proof of Main Results}
\begin{proof}[of \cref{thm:main}]
We will denote by $\preceq$ the lexicographic order on $\mathbb{N} \times \mathbb{N}$. With this we assume that $(\bG,F,\bL,\bP,\bM,\bQ)$ is a tuple such that the following hold:
\begin{enumerate}[label=(H\arabic*), leftmargin=1.4cm]
	\item $Z(\bG)$ is connected and $\bG$ has no quasi-simple component of type $\E_8$,
	\item $\Delta_{\bL\subset\bP,\bM\subset\bQ}^{\bG} \neq 0$,
	\item $(\dim\bG,\dim\bL+\dim\bM)$ is minimal, with respect to $\preceq$, amongst all the tuples satisfying (H1) and (H2).
\end{enumerate}
Arguing on the minimality of $(\dim\bG,\dim\bL+\dim\bM)$ we aim to show that such a tuple cannot exist. We follow precisely the argument used in the proof of \cite[3.9]{bonnafe-michel:2011:mackey-formula}.

As $p=2$ and (H1) holds there exist $F$-stable closed connected reductive subgroups $\bG_1,\bG_2\leqslant\bG$ such that the following hold:
\begin{itemize}
	\item all the quasi-simple components of $\bG_1$ are of type $\A$, $\E_6$, or $\E_7$,
	\item all the quasi-simple components of $\bG_2$ are of type $\B$, $\C$, $\D$, $\F_4$, or $\G_2$,
	\item the product map $\bG_1\times\bG_2 \to \bG$ is a bijective morphism of algebraic groups defined over $\mathbb{F}_q$.
\end{itemize}
As (H2) holds for $\bG$ we have by \cref{lem:bij-morph} that the same must be true of the direct product $\bG_1 \times \bG_2$. Now, by \cite[3.9]{bonnafe-michel:2011:mackey-formula}, the Mackey formula holds for $\bG_2$ so as Deligne--Lusztig induction is compatible with respect to direct products we can assume that the Mackey formula fails for $\bG_1$. Applying (H3) and \cref{lem:bij-morph} we may thus assume that all the quasi-simple components of $\bG$ are of type $\A$, $\E_6$, or $\E_7$.

Let us denote by $\mu \in \Class(\bM^F)$ a class function such that $\Delta_{\bL \subset \bP,\bM \subset \bQ}^{\bG}(\mu) \neq 0$. By \cite[3.2]{bonnafe-michel:2011:mackey-formula} there must exist a semisimple element $s \in \bL^F$ such that $d_s^{\bL}(\Delta_{\bL \subset \bP,\bM \subset \bQ}^{\bG}(\mu)) \neq 0$. Applying \cref{eq:d_s-formula} there thus exists an element $g \in \bG^F$ such that
\begin{equation*}
\Delta_{C_{\bL}^{\circ}(s)\subset C_{\bP}^{\circ}(s), C_{{}^g\bM}^{\circ}(s) \subset C_{{}^g\bQ}^{\circ}(s)}^{C_{\bG}^{\circ}(s)}({}^g\mu) \neq 0.
\end{equation*}
We set $\bM' = {}^g\bM$, $\bQ' = {}^g\bQ$ and $\lambda = d_s^{\bM'}({}^g\mu) \in \Class_{\uni}(C_{\bM'}^{\circ}(s)^F)$.

If $\bK \leqslant \bG$ is a closed subgroup of $\bG$ then we denote by $\bK_s$ the subgroup $C_{\bK}^{\circ}(s) \leqslant \bG$ and by $\bar{\bK}_s$ the image of $\bK_s$ under the natural quotient map $C_{\bG}^{\circ}(s) \to C_{\bG}^{\circ}(s)/Z^{\circ}(C_{\bG}^{\circ}(s))$. Note this quotient map is a surjective isotypic morphism with connected kernel. Therefore, by \cref{lem:inf-isom-uni}, there exists a unique unipotently supported class function $\bar{\lambda} \in \Class_{\uni}(\bar{\bG}_s^F)$ such that $\lambda = d_1^{\bG}(\Res_{\bG_s^F}^{\bar{\bG}_s^F}(\bar{\lambda}))$. Applying \cref{eq:res-formula,eq:d_s-centre-formula} we see that
\begin{equation*}
d_1^{\bM_s'}(\Res_{\bM_s'^F}^{\bar{\bM}_s'^F}(\Delta_{\bar{\bL}_s \subset \bar{\bP}_s, \bar{\bM}_s' \subset \bar{\bQ}_s'}^{\bar{\bG}_s}(\bar{\lambda})) = \Delta_{\bL_s \subset \bP_s,\bM_s' \subset \bQ_s'}^{\bG_s}(\lambda) \neq 0
\end{equation*}
so $\Delta_{\bar{\bL}_s \subset \bar{\bP}_s, \bar{\bM}_s' \subset \bar{\bQ}_s'}^{\bar{\bG}_s}(\bar{\lambda}) \neq 0$.

Let us now assume that $\bH_1,\bH_2 \leqslant \bar{\bG}_s$ are closed subgroups as in \cref{lem:cent-comps}. By \cite[3.9]{bonnafe-michel:2011:mackey-formula} we have the Mackey formula holds for $\bH_2$ so, arguing as above, we may assume the Mackey formula fails for $\bH_1$. Now, we have $\dim\bH_1 \leqslant \dim \bG_s \leqslant \dim \bG$ and $\bH_1$ satisfies (H1). Thus by (H3) we can assume these inequalities are equalities. In particular, this implies that $\bG_s = \bG$ and the quotient map $\bG \to \bG/Z^{\circ}(\bG) = \bG/Z(\bG)$ is bijective. Hence, we can assume that $Z(\bG)$ is trivial and $\mu \in \Class(\bM^F)$ is unipotently supported.

As $Z(\bG)$ is trivial we have by \cref{lem:bij-morph} that it is sufficient to consider the case where $\bG$ is adjoint so that $\bG$ is a direct product of its quasi-simple components. Moreover, by compatibility with direct products we can assume that $F$ cyclically permutes the quasi-simple components of $\bG$. Finally we can assume that either all the quasi-simple components are of type $\E_6$ or they are all of type $\E_7$ because the Mackey formula holds if they are of type $\A$ by \cite[3.9]{bonnafe-michel:2011:mackey-formula}.

Now let $(\bG^{\star},F^{\star})$ be a pair dual to $(\bG,F)$ and let $\bM^{\star} \leqslant \bG^{\star}$ be a Levi subgroup dual to $\bM \leqslant \bG$. We note that $\bG^{\star}$ is simply connected as $\bG$ is adjoint. Arguing exactly as in the proof of \cite[3.9]{bonnafe-michel:2011:mackey-formula} we may assume that the following properties hold:
\begin{enumerate}[label=(P\arabic*),start=3,leftmargin=1.4cm]
	\item $\bM$ is not a maximal torus and $\bM \neq \bG$,
	\item there exists an $F$-stable unipotent class of $\bM$ which supports an $F$-stable cuspidal local system, in the sense of \cite[2.4]{lusztig:1984:intersection-cohomology-complexes},
	\item $\bQ$ is not contained in an $F$-stable proper parabolic subgroup of $\bG$,
	\item there exists a semisimple element $s \in \bM^{\star F^{\star}}$ which is quasi-isolated in both $\bM^{\star}$ and $\bG^{\star}$ such that $sz$ is $\bG^{\star F^{\star}}$-conjugate to $s$ for every $z \in Z(\bM^{\star})^{F^{\star}}$.
\end{enumerate}
Indeed, (P3) follows immediately from the fact that the Mackey formula holds if either $\bL$ or $\bM$ is a maximal torus. Moreover, (P5) follows from the formula in \cite[3.4]{bonnafe-michel:2011:mackey-formula} together with the fact that the Mackey formula holds if both $\bP$ and $\bQ$ are $F$-stable. The remaining properties (P4) and (P6) are established by using the fact that (H1) holds for all proper Levi subgroups of $\bG$. In particular, the Mackey formula holds for all proper Levi subgroups of $\bG$.

It is already established in \cite[Lemma ($\E7$)]{bonnafe-michel:2011:mackey-formula} that if the quasi-simple components of $\bG$ are of type $\E_7$ then there is no pair $(\bM,\bQ)$ satisfying (P3) to (P6). Hence we can assume that all the quasi-simple components are of type $\E_6$. As $\bG$ is adjoint and $p=2$ the only possible choice for $\bM$ satisfying (P3) and (P4) is a Levi subgroup of type $\D_4$, see \cite[15.1]{lusztig:1984:intersection-cohomology-complexes}. However the exact same argument used in the proof of \cite[2.2(f)]{bonnafe-michel:2011:mackey-formula} shows that no such Levi subgroup can satisfy both (P5) and (P6). This completes the proof.
\end{proof}

\begin{proof}[of \cref{thm:main-unip}]
Assume for a contradiction that $\mu \in \Class_{\uni}(\bM^F)$ is a unipotently supported class function satisfying $\Delta_{\bL\subset\bP,\bM\subset\bQ}^{\bG}(\mu) \neq 0$. By \cite[3.9]{bonnafe-michel:2011:mackey-formula} we can assume that $q = 2$. By \cref{thm:main,lem:bij-morph} and compatibility with direct products we can assume that all the quasi-simple components of $\bG$ are of type $\E_8$ and that $F$ cyclically permutes these quasi-simple components. Note that $\bG$ is necessarily semisimple and simply connected.

We note that any proper $F$-stable Levi subgroup of $\bG$ has connected centre and has no quasi-simple component of type $\E_8$. Thus by \cref{thm:main} the Mackey formula holds for any proper $F$-stable Levi subgroup. With this we may argue as above, and exactly as in the proof of \cite[3.9]{bonnafe-michel:2011:mackey-formula}, that the pair $(\bM,\bQ)$ satisfies the properties (P1) to (P6) of \cite[2.1]{bonnafe-michel:2011:mackey-formula}. However, \cite[2.1]{bonnafe-michel:2011:mackey-formula} establishes precisely that there is no such pair $(\bM,\bQ)$ satisfying these properties, so we must have $\Delta_{\bL\subset\bP,\bM\subset\bQ}^{\bG}(\mu) = 0$.
\end{proof}

\begingroup
\setstretch{0.96}
\renewcommand*{\bibfont}{\small}
\printbibliography
\endgroup
\end{document}